\documentclass{amsart}
\usepackage[dvipdfmx]{graphicx}
\newtheorem{thm}{Theorem}[section]
\newtheorem{lem}[thm]{Lemma}
\newtheorem{prop}[thm]{Proposition}
\newtheorem{cor}[thm]{Corollary}

\theoremstyle{definition}
\newtheorem{defn}[thm]{Definition}
\newtheorem{ex}[thm]{Example}
\theoremstyle{remark}
\newtheorem{rem}[thm]{Remark}
\newcommand{\A}{\mathcal{A}}
\newcommand{\F}{\mathcal{F}}
\newcommand{\FF}{\mathbb{F}}
\newcommand{\Q}{\mathbb{Q}}
\newcommand{\R}{\mathbb{R}}
\newcommand{\SSS}{\mathcal{S}}
\newcommand{\Z}{\mathbb{Z}}
\newcommand{\ZZ}{\mathfrak{Z}}
\newcommand{\ZZZ}{\mathcal{Z}}
\newcommand{\sym}{\mathfrak{S}}
\newcommand{\ve}[1]{\boldsymbol{#1}}
\makeatletter
\DeclareRobustCommand{\sbinom}{\genfrac[]\z@{}}
\makeatother
\begin{document}
\title{Restricted sum formula for finite and symmetric multiple zeta values}
\author{Hideki Murahara}
\address{Nakamura Gakuen University Graduate School,
 5-7-1, Befu, Jonan-ku, Fukuoka, 814-0198, Japan} 
\email{hmurahara@nakamura-u.ac.jp}
\author{Shingo Saito}
\address{Faculty of Arts and Science, Kyushu University,
 744, Motooka, Nishi-ku, Fukuoka, 819-0395, Japan}
\email{ssaito@artsci.kyushu-u.ac.jp}
\begin{abstract}
 The sum formula for finite and symmetric multiple zeta values, established by Wakabayashi and the authors,
 implies that if the weight and depth are fixed and the specified component is required to be more than one,
 then the values sum up to a rational multiple of the analogue of the Riemann zeta value.
 We prove that the result remains true if we further demand
 that the component should be more than two or that another component should also be more than one.
\end{abstract}
\maketitle

\section{Introduction}
The \emph{multiple zeta values} and \emph{multiple zeta-star values} are the real numbers defined by
\begin{align*}
 \zeta(k_1,\dots,k_r)&=\sum_{m_1>\dots>m_r\ge1}\frac{1}{m_1^{k_1}\dotsm m_r^{k_r}},\\
 \zeta^{\star}(k_1,\dots,k_r)&=\sum_{m_1\ge\dots\ge m_r\ge1}\frac{1}{m_1^{k_1}\dotsm m_r^{k_r}}
\end{align*}
for $k_1,\dots,k_r\in\Z_{\ge1}$ with $k_1\ge2$.
They are generalisations of the values of the Riemann zeta function at positive integers,
and they are known to have interesting algebraic structures due to the many relations among them,
the simplest being $\zeta(2,1)=\zeta(3)$.
See, for example, the book \cite{Zh} by Zhao for further details on multiple zeta(-star) values.

The variants of multiple zeta values that we shall be looking at in this paper
are \emph{finite multiple zeta values} $\zeta_{\A}(k_1,\dots,k_r)$ and
\emph{symmetric multiple zeta values} $\zeta_{\SSS}(k_1,\dots,k_r)$
(the latter also known as symmetrised multiple zeta values and finite real multiple zeta values),
both introduced by Kaneko and Zagier \cite{KZ} (see \cite{Zh} for details).
Set $\A=\prod_p\FF_p/\bigoplus_p\FF_p$, where $p$ runs over all primes.
For $k_1,\dots,k_r\in\Z_{\ge1}$, we define
\begin{align*}
 \zeta_{\A}(k_1,\dots,k_r)&=\Biggl(\sum_{p>m_1>\dots>m_r\ge1}\frac{1}{m_1^{k_1}\dotsm m_r^{k_r}}\bmod p\Biggr)_p\in\A,\\
 \zeta_{\A}^{\star}(k_1,\dots,k_r)&=\Biggl(\sum_{p>m_1\ge\dots\ge m_r\ge1}\frac{1}{m_1^{k_1}\dotsm m_r^{k_r}}\bmod p\Biggr)_p\in\A.
\end{align*}
Let $\ZZZ$ denote the $\Q$-linear subspace of $\R$ spanned by the multiple zeta values.
For $k_1,\dots,k_r\in\Z_{\ge1}$, we define
\begin{align*}
 \zeta_{\SSS}(k_1,\dots,k_r)&=\sum_{j=0}^{r}(-1)^{k_1+\dots+k_j}\zeta(k_j,\dots,k_1)\zeta(k_{j+1},\dots,k_r)\bmod\zeta(2)\in\ZZZ/\zeta(2)\ZZZ,\\
 \zeta_{\SSS}^{\star}(k_1,\dots,k_r)&=\sum_{j=0}^{r}(-1)^{k_1+\dots+k_j}\zeta^{\star}(k_j,\dots,k_1)\zeta^{\star}(k_{j+1},\dots,k_r)\bmod\zeta(2)\in\ZZZ/\zeta(2)\ZZZ,
\end{align*}
where we set $\zeta(\emptyset)=\zeta^{\star}(\emptyset)=1$.
The multiple zeta(-star) values that appear in the definition of the symmetric multiple zeta(-star) values
are the regularised values if the first component is $1$;
although there are two ways of regularisation, called the harmonic regularisation and the shuffle regularisation,
it is known that the symmetric multiple zeta values remain unchanged as elements of $\ZZZ/\zeta(2)\ZZZ$
no matter which regularisation we use (see \cite{KZ}).

Kaneko and Zagier \cite{KZ} made a striking conjecture that the finite multiple zeta values
and the symmetric multiple zeta values are isomorphic;
more precisely, if we let $\ZZZ_{\A}$ denote the $\Q$-linear subspace of $\A$ spanned by the finite multiple zeta values,
then $\ZZZ_{\A}$ and $\ZZZ/\zeta(2)\ZZZ$ are isomorphic as $\Q$-algebras via the correspondence
$\zeta_{\A}(k_1,\dots,k_r)\leftrightarrow\zeta_{\SSS}(k_1,\dots,k_r)$.
It means that $\zeta_{\A}(k_1,\dots,k_r)$ and $\zeta_{\SSS}(k_1,\dots,k_r)$ satisfy the same relations,
and a notable example of such relations is the sum formula (Theorem~\ref{thm:sum_formula}).
In what follows, we use the letter $\F$ when it can be replaced with either $\A$ or $\SSS$;
for example, by $\zeta_{\F}(1)=0$ we mean that both $\zeta_{\A}(1)=0$ and $\zeta_{\SSS}(1)=0$ are true.
We write
\[
 \ZZ_{\F}(k)=
 \begin{cases}
  (B_{p-k}/k\bmod p)_p&\text{if $\F=\A$;}\\
  \zeta(k)\bmod\zeta(2)&\text{if $\F=\SSS$}
 \end{cases}
\]
for $k\in\Z_{\ge2}$, where $B_n$ denotes the $n$-th Bernoulli number.
Note that it can be verified rather easily that $\zeta_{\F}(k-1,1)=\ZZ_{\F}(k)$ for $k\in\Z_{\ge2}$,
so that $(B_{p-k}/k\bmod p)_p$ corresponds to $\zeta(k)\bmod\zeta(2)$ via the above-mentioned isomorphism $\ZZZ_{\A}\cong\ZZZ/\zeta(2)\ZZZ$.

\begin{thm}[Saito-Wakabayashi {\cite{SW}}, Murahara {\cite{M}}]\label{thm:sum_formula}
 For $k,r,i\in\Z$ with $1\le i\le r\le k-1$, we have
 \begin{align*}
  \sum_{\substack{k_1+\dots+k_r=k\\k_i\ge2}}\zeta_{\F}(k_1,\dots,k_r)
  &=(-1)^r\sum_{\substack{k_1+\dots+k_r=k\\k_i\ge2}}\zeta_{\F}^{\star}(k_1,\dots,k_r)\\
  &=(-1)^{i-1}\biggl(\binom{k-1}{i-1}+(-1)^r\binom{k-1}{r-i}\biggr)\ZZ_{\F}(k).
 \end{align*}
\end{thm}

The theorem implies that the sums belong to $\Q\ZZ_{\F}(k)$.
Our main theorem states that similar sums also belong to $\Q\ZZ_{\F}(k)$ if $k$ is odd:
\begin{thm}[Main theorem]\label{thm:main}
 Let $k$ be an odd integer with $k\ge3$, and let $r$ be an integer with $1\le r\le k-2$.
 \begin{enumerate}
  \item For $i\in\Z$ with $1\le i\le r$, we have
   \[
    \sum_{\substack{k_1+\dots+k_r=k\\k_i\ge3}}\zeta_{\F}(k_1,\dots,k_r)
    =(-1)^r\sum_{\substack{k_1+\dots+k_r=k\\k_i\ge3}}\zeta_{\F}^{\star}(k_1,\dots,k_r)
    \in\Q\ZZ_{\F}(k).
   \]
  \item For distinct $i,j\in\Z$ with $1\le i,j\le r$, we have
   \[
    \sum_{\substack{k_1+\dots+k_r=k\\k_i,k_j\ge2}}\zeta_{\F}(k_1,\dots,k_r)
    =(-1)^r\sum_{\substack{k_1+\dots+k_r=k\\k_i,k_j\ge2}}\zeta_{\F}^{\star}(k_1,\dots,k_r)
    \in\Q\ZZ_{\F}(k).
   \]
 \end{enumerate}
\end{thm}

The rational coefficients can be written explicitly, though in a rather complicated manner,
in terms of binomial coefficients (see Theorem~\ref{thm:main_explict} for the preciese statement).

\begin{rem}
 If $k$ is even, then $\ZZ_{\F}(k)=0$ and numerical experiments suggest that the sums are not always equal to $0$.
\end{rem}

\section{Preliminary lemmas}
This section will give a few preliminary lemmas that will be used to prove our main theorem in the next section.

An \emph{index} is a (possibly empty) sequence of positive integers.
For an index $\ve{k}=(k_1,\dots,k_r)$,
the number $r$ is called its \emph{depth} and $k_1+\dots+k_r$ its \emph{weight}.

\begin{prop}\label{prop:sym_sum}
 If $(k_1,\dots,k_r)$ is a nonempty index, then
 \[
  \sum_{\sigma\in\sym_r}\zeta_{\F}(k_{\sigma(1)},\dots,k_{\sigma(r)})=\sum_{\sigma\in\sym_r}\zeta_{\F}^{\star}(k_{\sigma(1)},\dots,k_{\sigma(r)})=0,
 \]
 where $\sym_r$ denotes the symmetric group of order $r$.
\end{prop}

\begin{proof}
 Roughly speaking, the sums are zero because they can be written as polynomials of the values $\zeta_{\F}(k)$, which are all zero.
 For details, see \cite[Theorem~2.3]{H} and \cite[Proposition~2.7]{S}, for example.
\end{proof}

We write $\{k\}^r$ for the $r$ times repetition of $k$.

\begin{cor}\label{cor:rep_0}
 For $k,r\in\Z_{\ge1}$, we have
 \[
  \zeta_{\F}(\{k\}^r)=\zeta_{\F}^{\star}(\{k\}^r)=0.
 \]
\end{cor}

\begin{proof}
 Apply Proposition~\ref{prop:sym_sum} to $(k_1,\dots,k_r)=(\{k\}^r)$.
\end{proof}

\begin{defn}
 For each index $\ve{k}$, write its components as sums of ones,
 and define its \emph{Hoffman dual} $\ve{k}^{\vee}$ as the index obtained by swapping plus signs and commas.
\end{defn}

\begin{ex}
 If $\ve{k}=(2,1,3)=(1+1,1,1+1+1)$, then $\ve{k}^{\vee}=(1,1+1+1,1,1)=(1,3,1,1)$.
\end{ex}

The following theorem, known as \emph{duality}, was proved by Hoffman~\cite{H} for the $\F=\A$ case
and by Jarossay~\cite{J} for the $\F=\SSS$ case:
\begin{thm}[Hoffman~\cite{H}, Jarossay~\cite{J}]\label{thm:duality}
 If $\ve{k}$ is a nonempty index, then
 \[
  \zeta_{\F}^{\star}(\ve{k}^{\vee})=-\zeta_{\F}^{\star}(\ve{k}).
 \]
\end{thm}

For indices $\ve{k}$ and $\ve{l}$ of the same weight,
we write $\ve{k}\preceq\ve{l}$ to mean that, writing their components as sums of ones,
we can obtain $\ve{l}$ from $\ve{k}$ by replacing some (possibly none) of the plus signs with commas.
For example, $(2,1,3)=(1+1,1,1+1+1)\preceq(1,1,1,1+1,1)=(1,1,1,2,1)$.

\begin{cor}\label{prop:phi_dual}
 If $\ve{k}$ is a nonempty index of depth $r$, then
 \[
  (-1)^r\zeta_{\F}(\ve{k})=\sum_{\ve{l}\succeq\ve{k}}\zeta_{\F}(\ve{l}).
 \]
\end{cor}

\begin{proof}
 An easy combinatorial argument shows that this corollary is equivalent to Theorem~\ref{thm:duality};
 see \cite[Corollary~2.15]{S} for details.
\end{proof}

We adopt the standard convention for binomial coefficients that $\binom{a}{b}=0$ if $a\in\Z_{\ge0}$ and $b\in\Z\setminus\{0,\dots,a\}$.
For notational simplicity, we write
\[
 \sbinom{a}{b}=(-1)^b\binom{a}{b}
\]
for $a\in\Z_{\ge0}$ and $b\in\Z$
(not to be confused with the Stirling numbers of the first kind).
Then Theorem~\ref{thm:sum_formula} can be rewritten as follows:
\begin{thm}[Another form of Theorem~\ref{thm:sum_formula}]\label{thm:sum_formula2}
 For $k,r,i\in\Z$ with $1\le i\le r\le k-1$, we have
 \begin{align*}
  \sum_{\substack{k_1+\dots+k_r=k\\k_i\ge2}}\zeta_{\F}(k_1,\dots,k_r)
  &=(-1)^r\sum_{\substack{k_1+\dots+k_r=k\\k_i\ge2}}\zeta_{\F}^{\star}(k_1,\dots,k_r)\\
  &=\biggl(\sbinom{k-1}{i-1}-\sbinom{k-1}{r-i}\biggr)\ZZ_{\F}(k).
 \end{align*}
\end{thm}

\begin{lem}\label{lem:zeta_1_2_1}
 For $a,b\in\Z_{\ge0}$ with $a+b$ odd, we have
 \[
  \zeta_{\F}(\{1\}^a,2,\{1\}^b)=-\sbinom{a+b+2}{a+1}\ZZ_{\F}(a+b+2)=\sbinom{a+b+2}{b+1}\ZZ_{\F}(a+b+2).
 \]
\end{lem}

\begin{proof}
 Applying Theorem~\ref{thm:sum_formula2} to $k=a+b+2$, $r=a+b+1$, and $i=a+1$ gives
 \[
  \zeta_{\F}(\{1\}^a,2,\{1\}^b)=\biggl(\sbinom{a+b+1}{a}-\sbinom{a+b+1}{b}\biggr)\ZZ_{\F}(a+b+2),
 \]
 and we have
 \begin{align*}
  \sbinom{a+b+1}{a}-\sbinom{a+b+1}{b}
  &=(-1)^{a}\binom{a+b+1}{a}-(-1)^b\binom{a+b+1}{b}\\
  &=-(-1)^{a+1}\biggl(\binom{a+b+1}{a}+\binom{a+b+1}{a+1}\biggr)\\
  &=-(-1)^{a+1}\binom{a+b+2}{a+1}\\
  &=-\sbinom{a+b+2}{a+1}.
 \end{align*}
 By a similar reasoning, we also have
 \[
  \sbinom{a+b+1}{a}-\sbinom{a+b+1}{b}=\sbinom{a+b+2}{b+1}.\qedhere
 \]
\end{proof}

\begin{lem}\label{lem:acb}
 For $a,b\in\Z_{\ge0}$ and $c\in\Z_{\ge-1}$ with $a+b+c$ odd, we have
 \[
  \zeta_{\F}(\{1\}^a,2,\{1\}^c,2,\{1\}^b)=\frac{1}{2}\biggl(\sbinom{a+b+c+4}{a+1}-\sbinom{a+b+c+4}{b+1}\biggr)\ZZ_{\F}(a+b+c+4),
 \]
 where we understand that $\zeta_{\F}(\{1\}^a,2,\{1\}^{-1},2,\{1\}^b)=\zeta_{\F}(\{1\}^a,3,\{1\}^b)$.
\end{lem}

\begin{proof}
 Keeping Corollary~\ref{cor:rep_0} in mind, we apply Proposition~\ref{prop:phi_dual} to $\ve{k}=(\{1\}^a,2,\{1\}^c,2,\{1\}^b)$ to get
 \begin{align*}
  &-\zeta_{\F}(\{1\}^a,2,\{1\}^c,2,\{1\}^b)\\
  &\qquad=\zeta_{\F}(\{1\}^a,2,\{1\}^c,2,\{1\}^b)+\zeta_{\F}(\{1\}^a,2,\{1\}^{b+c+2})+\zeta_{\F}(\{1\}^{a+c+2},2,\{1\}^b),
 \end{align*}
 no matter whether $c=-1$ or $c\ge0$.
 This, together with Lemma~\ref{lem:zeta_1_2_1}, gives
 \begin{align*}
  &\zeta_{\F}(\{1\}^a,2,\{1\}^c,2,\{1\}^b)\\
  &\qquad=-\frac{1}{2}(\zeta_{\F}(\{1\}^a,2,\{1\}^{b+c+2})+\zeta_{\F}(\{1\}^{a+c+2},2,\{1\}^b))\\
  &\qquad=\frac{1}{2}\biggl(\sbinom{a+b+c+4}{a+1}-\sbinom{a+b+c+4}{b+1}\biggr)\ZZ_{\F}(a+b+c+4).\qedhere
 \end{align*}
\end{proof}

\section{Proof of the main theorem}
Throughout this section, let $k$ be an odd integer with $k\ge3$, and let $r$, $i$, $j$ be integers with $1\le i\le j\le r\le k-2$.
Set
\[
 I_{k,r,i,j}=
 \begin{cases}
  \{(k_1,\dots,k_r)\in\Z_{\ge1}^r\mid k_i\ge3\}&\text{if $i=j$;}\\
  \{(k_1,\dots,k_r)\in\Z_{\ge1}^r\mid k_i,k_j\ge2\}&\text{if $i<j$,} 
 \end{cases}
\]
and write
\[
 S_{k,r,i,j}=\sum_{\ve{k}\in I_{k,r,i,j}}\zeta_{\F}(\ve{k}),\qquad
 S_{k,r,i,j}^{\star}=\sum_{\ve{k}\in I_{k,r,i,j}}\zeta_{\F}^{\star}(\ve{k}).
\]
For notational simplicity, we put $i'=j-i+1$, $i''=r-j+1$, and $k'=k-r-2$, so that $i+i'+i''+k'=k$.

The aim of this section is to prove the following theorem, from which Theorem~\ref{thm:main} easily follows:
\begin{thm}\label{thm:main_explict}
 We have
 \[
  S_{k,r,i,j}=(-1)^rS_{k,r,i,j}^{\star}=\frac{1}{2}N_{k,r,i,j}\ZZ_{\F}(k).
 \]
 where $N_{k,r,i,j}$ is an integer given by
 \begin{align*}
  N_{k,r,i,j}
  &=(k'+i+1)\biggl(\sbinom{k-1}{k'+i}-\sbinom{k-1}{i-1}\biggr)-(k'+i''+1)\biggl(\sbinom{k-1}{k'+i''}-\sbinom{k-1}{i''-1}\biggr)\\
  &\qquad+k\biggl(\sbinom{k-2}{k'+i-1}-\sbinom{k-2}{i-2}-\sbinom{k-2}{k'+i''-1}+\sbinom{k-2}{i''-2}\biggr).
 \end{align*}
\end{thm}

\subsection{Proof that $S_{k,r,i,j}=(-1)^rS_{k,r,i,j}^{\star}$}
In this subsection, we shall prove that $S_{k,r,i,j}=(-1)^rS_{k,r,i,j}^{\star}$ (Lemma~\ref{lem:S_krij_star}).

\begin{prop}\label{prop:reversal}
 If $(k_1,\dots,k_r)$ is an index, then
 \begin{align*}
  \zeta_{\F}(k_r,\dots,k_1)&=(-1)^{k_1+\dots+k_r}\zeta_{\F}(k_1,\dots,k_r),\\
  \zeta_{\F}^{\star}(k_r,\dots,k_1)&=(-1)^{k_1+\dots+k_r}\zeta_{\F}^{\star}(k_1,\dots,k_r).
 \end{align*}
\end{prop}

\begin{proof}
 Easy from the definitions; see \cite[Proposition~2.6]{S} for details.
\end{proof}

\begin{prop}\label{prop:antipode}
 If $\ve{k}=(k_1,\dots,k_r)$ is a nonempty index, then
 \[
  \sum_{s=0}^{r}(-1)^s\zeta_{\F}^{\star}(k_s,\dots,k_1)\zeta_{\F}(k_{s+1},\dots,k_r)=0,
 \]
 where we set $\zeta_{\F}(\emptyset)=\zeta_{\F}^{\star}(\emptyset)=1$.
\end{prop}

\begin{proof}
 Well known; see \cite[Proposition~2.9]{S} for the detailed proof.
\end{proof}

\begin{lem}\label{lem:S_krij_star}
 We have
 \[
  S_{k,r,i,j}=(-1)^rS_{k,r,i,j}^{\star}.
 \]
\end{lem}

\begin{proof}
 Adding the equation in Proposition~\ref{prop:antipode} for all $(k_1,\dots,k_r)\in I_{k,r,i,j}$ gives
 \[
  \sum_{s=0}^{r}(-1)^s\sum_{(k_1,\dots,k_r)\in I_{k,r,i,j}}\zeta_{\F}^{\star}(k_s,\dots,k_1)\zeta_{\F}(k_{s+1},\dots,k_r)=0,
 \]
 whose left-hand side we shall write as $\sum_{s=0}^{r}(-1)^sA_s$ for simplicity.
 Observe that $A_0=S_{k,r,i,j}$ and that
 \begin{align*}
  A_r&=\sum_{(k_1,\dots,k_r)\in I_{k,r,i,j}}\zeta_{\F}^{\star}(k_r,\dots,k_1)\\
  &=\sum_{(k_1,\dots,k_r)\in I_{k,r,i,j}}(-1)^{k_1+\dots+k_r}\zeta_{\F}^{\star}(k_1,\dots,k_r)\\
  &=-S_{k,r,i,j}^{\star}
 \end{align*}
 by Proposition~\ref{prop:reversal} because $k$ is odd.
 For $s=j,\dots,r-1$, we have
 \begin{align*}
  A_s&=\sum_{l=0}^{k}\Biggl(\sum_{(k_1,\dots,k_s)\in I_{l,s,i,j}}\zeta_{\F}^{\star}(k_s,\dots,k_1)\Biggr)\Biggl(\sum_{k_{s+1}+\dots+k_r=k-l}\zeta_{\F}(k_{s+1},\dots,k_r)\Biggr)\\
  &=0
 \end{align*}
 because of Proposition~\ref{prop:sym_sum};
 we similarly have $A_s=0$ for $s=1,\dots,i-1$.
 If $i<j$ and $i\le s\le j-1$, then we have
 \begin{align*}
  A_s&=\sum_{l=0}^{k}\Biggl(\sum_{\substack{k_1+\dots+k_s=l\\k_i\ge2}}\zeta_{\F}^{\star}(k_s,\dots,k_1)\Biggr)\Biggl(\sum_{\substack{k_{s+1}+\dots+k_r=k-l\\k_j\ge2}}\zeta_{\F}(k_{s+1},\dots,k_r)\Biggr)\\
  &=\sum_{l=0}^{k}\Biggl((-1)^l\sum_{\substack{k_1+\dots+k_s=l\\k_i\ge2}}\zeta_{\F}^{\star}(k_1,\dots,k_s)\Biggr)\Biggl(\sum_{\substack{k_{s+1}+\dots+k_r=k-l\\k_j\ge2}}\zeta_{\F}(k_{s+1},\dots,k_r)\Biggr)\\
  &=\sum_{l=0}^{k}(-1)^{l+s}\biggl(\sbinom{l-1}{i-1}-\sbinom{l-1}{s-i}\biggr)\ZZ_{\F}(l)\biggl(\sbinom{k-l-1}{j-s-1}-\sbinom{k-l-1}{r-j}\biggr)\ZZ_{\F}(k-l)
 \end{align*}
 by Proposition~\ref{prop:reversal} and Theorem~\ref{thm:sum_formula2};
 since $k$ is odd, either $l$ or $k-l$ must even and so $\ZZ_{\F}(l)\ZZ_{\F}(k-l)=0$ for all $l=0,\dots,k$, from which it follows that $A_s=0$.
 Therefore we have $S_{k,r,i,j}-(-1)^rS_{k,r,i,j}^{\star}=0$, and the lemma follows.
\end{proof}

\subsection{Computation of $S_{k,r,i,j}$}
In this subsection, we shall compute $S_{k,r,i,j}$ (Lemma~\ref{lem:S_krij}).
The main ingredient of the computation is the following Ohno type relation, conjectured by Kaneko~\cite{K} and established by Oyama~\cite{Oy}:
\begin{thm}[Oyama~{\cite[Theorem~1.4]{Oy}}]\label{thm:oyama}
 Let $\ve{k}=(k_1,\dots,k_r)$ be an index, and write its Hoffman dual as $\ve{k}^{\vee}=(k_1',\dots,k_{r'}')$.
 Then for $m\in\Z_{\ge0}$, we have
 \[
  \sum_{\substack{e_1+\dots+e_r=m\\e_1,\dots,e_r\ge0}}\zeta_{\F}(k_1+e_1,\dots,k_r+e_r)
  =\sum_{\substack{e_1'+\dots+e_{r'}'=m\\e_1',\dots,e_r'\ge0}}\zeta_{\F}((k_1'+e_1',\dots,k_{r'}'+e_{r'}')^{\vee}).
 \]
\end{thm}

\begin{lem}\label{lem:use_oyama}
 We have
 \[
  S_{k,r,i,j}=\sum_{\substack{e_1'+e_2'+e_3'=k'\\e_1',e_2',e_3'\ge0}}\zeta_{\F}((i+e_1',i'+e_2',i''+e_3')^{\vee}).
 \]
\end{lem}

\begin{proof}
 Theorem~\ref{thm:oyama} shows that if $i=j$, then
 \begin{align*}
  S_{k,r,i,j}
  &=\sum_{\substack{e_1+\dots+e_r=k'\\e_1,\dots,e_r\ge0}}\zeta_{\F}(1+e_1,\dots,1+e_{i-1},3+e_i,1+e_{i+1},\dots,1+e_r)\\
  &=\sum_{\substack{e_1'+e_2'+e_3'=k'\\e_1',e_2',e_3'\ge0}}\zeta_{\F}((i+e_1',i'+e_2',i''+e_3')^{\vee}),
 \end{align*}
 and that if $i<j$, then
 \begin{align*}
  &S_{k,r,i,j}\\
  &=\sum_{\substack{e_1+\dots+e_r=k'\\e_1,\dots,e_r\ge0}}\zeta_{\F}(1+e_1,\dots,1+e_{i-1},2+e_i,1+e_{i+1},\dots,1+e_{j-1},2+e_j,1+e_{j+1},\dots,1+e_r)\\
  &=\sum_{\substack{e_1'+e_2'+e_3'=k'\\e_1',e_2',e_3'\ge0}}\zeta_{\F}((i+e_1',i'+e_2',i''+e_3')^{\vee}).\qedhere
 \end{align*}
\end{proof}

\begin{lem}\label{lem:S_krij_sum}
 We have
 \[
  S_{k,r,i,j}=\frac{1}{2}\sum_{\substack{e_1'+e_2'+e_3'=k'\\e_1',e_2',e_3'\ge0}}\biggl(\sbinom{k}{i+e_1'}-\sbinom{k}{i''+e_3'}\biggr)\ZZ_{\F}(k).
 \]
\end{lem}

\begin{proof}
 Using the same convention as in the statement of Lemma~\ref{lem:acb}, we have
 \[
  (i+e_1',i'+e_2',i''+e_3')^{\vee}=(\{1\}^{i+e_1'-1},2,\{1\}^{i'+e_2'-2},2,\{1\}^{i''+e_3'-1}),
 \]
 and so by Lemmas~\ref{lem:acb} and \ref{lem:use_oyama}, we have
 \begin{align*}
  S_{k,r,i,j}
  &=\sum_{\substack{e_1'+e_2'+e_3'=k'\\e_1',e_2',e_3'\ge0}}\zeta_{\F}((i+e_1',i'+e_2',i''+e_3')^{\vee})\\
  &=\sum_{\substack{e_1'+e_2'+e_3'=k'\\e_1',e_2',e_3'\ge0}}\zeta_{\F}(\{1\}^{i+e_1'-1},2,\{1\}^{i'+e_2'-2},2,\{1\}^{i''+e_3'-1})\\
  &=\frac{1}{2}\sum_{\substack{e_1'+e_2'+e_3'=k'\\e_1',e_2',e_3'\ge0}}\biggl(\sbinom{k}{i+e_1'}-\sbinom{k}{i''+e_3'}\biggr)\ZZ_{\F}(k).\qedhere
 \end{align*}
\end{proof}

\begin{lem}\label{lem:sum_signed_binom}
 We have
 \begin{align*}
  \sum_{\substack{e_1'+e_2'+e_3'=k'\\e_1',e_2',e_3'\ge0}}\sbinom{k}{i+e_1'}
  &=(k'+i+1)\biggl(\sbinom{k-1}{k'+i}-\sbinom{k-1}{i-1}\biggr)+k\biggl(\sbinom{k-2}{k'+i-1}-\sbinom{k-2}{i-2}\biggr),\\
  \sum_{\substack{e_1'+e_2'+e_3'=k'\\e_1',e_2',e_3'\ge0}}\sbinom{k}{i''+e_3'}
  &=(k'+i''+1)\biggl(\sbinom{k-1}{k'+i''}-\sbinom{k-1}{i''-1}\biggr)+k\biggl(\sbinom{k-2}{k'+i''-1}-\sbinom{k-2}{i''-2}\biggr).
 \end{align*}
\end{lem}

\begin{proof}
 By symmetry, we only need to show the first equality, which can be seen as follows:
 \begin{align*}
  &\sum_{\substack{e_1'+e_2'+e_3'=k'\\e_1',e_2',e_3'\ge0}}\sbinom{k}{i+e_1'}\\
  &\qquad=\sum_{e_1'=0}^{k'}(-1)^{i+e_1'}(k'-e_1'+1)\binom{k}{i+e_1'}\\
  &\qquad=\sum_{e_1'=0}^{k'}(-1)^{i+e_1'}((k'+i+1)-(i+e_1'))\binom{k}{i+e_1'}\\
  &\qquad=(k'+i+1)\sum_{e_1'=0}^{k'}(-1)^{i+e_1'}\binom{k}{i+e_1'}-k\sum_{e_1'=0}^{k'}(-1)^{i+e_1'}\binom{k-1}{i+e_1'-1}\\
  &\qquad=(k'+i+1)\sum_{e_1'=0}^{k'}\biggl((-1)^{i+e_1'}\binom{k-1}{i+e_1'}-(-1)^{i+e_1'-1}\binom{k-1}{i+e_1'-1}\biggr)\\
  &\qquad\qquad+k\sum_{e_1'=0}^{k'}\biggl((-1)^{i+e_1'-1}\binom{k-2}{i+e_1'-1}-(-1)^{i+e_1'-2}\binom{k-2}{i+e_1'-2}\biggr)\\
  &\qquad=(k'+i+1)\biggl((-1)^{k'+i}\binom{k-1}{k'+i}-(-1)^{i-1}\binom{k-1}{i-1}\biggr)\\
  &\qquad\qquad+k\biggl((-1)^{k'+i-1}\binom{k-2}{k'+i-1}-(-1)^{i-2}\binom{k-2}{i-2}\biggr)\\
  &\qquad=(k'+i+1)\biggl(\sbinom{k-1}{k'+i}-\sbinom{k-1}{i-1}\biggr)+k\biggl(\sbinom{k-2}{k'+i-1}-\sbinom{k-2}{i-2}\biggr).\qedhere
 \end{align*}
\end{proof}

\begin{lem}\label{lem:S_krij}
 We have
 \[
  S_{k,r,i,j}=\frac{1}{2}N_{k,r,i,j}\ZZ_{\F}(k).
 \]
 where $N_{k,r,i,j}$ is an integer given by
 \begin{align*}
  N_{k,r,i,j}
  &=(k'+i+1)\biggl(\sbinom{k-1}{k'+i}-\sbinom{k-1}{i-1}\biggr)-(k'+i''+1)\biggl(\sbinom{k-1}{k'+i''}-\sbinom{k-1}{i''-1}\biggr)\\
  &\qquad+k\biggl(\sbinom{k-2}{k'+i-1}-\sbinom{k-2}{i-2}-\sbinom{k-2}{k'+i''-1}+\sbinom{k-2}{i''-2}\biggr).
 \end{align*}
\end{lem}

\begin{proof}
 Immediate from Lemmas~\ref{lem:S_krij_sum} and \ref{lem:sum_signed_binom}.
\end{proof}

Lemmas~\ref{lem:S_krij_star} and \ref{lem:S_krij} complete the proof of our main theorem (Theorem~\ref{thm:main_explict}).

\end{document}